%% file: main.tex
\documentclass[reqno]{amsart}

\input{Preamble}
 \usetikzlibrary{math}
 \usetikzlibrary{backgrounds}
 \newcount\recurdepth
 \newcount\depth
\newcommand{\iecon}{infi\-nite\-ly edge-con\-nec\-ted}

\newcommand{\TTT}{T_{\aleph_0}\!\ast t}
\newcommand{\minor}{\preccurlyeq}
\newcommand{\rminor}{\succcurlyeq}
\newcommand{\fcut}{{\mathcal{B}}_{\aleph_0}}
\def\vcB{\lowfwd \cB{1.5}1}
\newcommand{\vfcut}{\vcB_{\mkern-.65\thinmuskip\aleph_0}}
\newcommand{\FG}{F}
\newcommand{\hFG}{\breve{F}}
\newcommand{\con}{/\!\!/}

\newcommand{\fsep}[4]{\bigl(\begin{smallmatrix}#1 & #4\\ #2 & #3\end{smallmatrix}\bigr)}

\begin{document}

\title[Infinite edge-connectivity, the Farey graph and $\TTT$]{Every infinitely edge-connected graph contains the Farey graph or $\boldsymbol{\TTT}$ as a minor}

\author{Jan Kurkofka}
\address{University of Hamburg, Department of Mathematics, Bundesstraße 55 (Geomatikum), 20146 Hamburg, Germany}
\email{jan.kurkofka@uni-hamburg.de}

\begin{abstract}
We show that every \iecon\ graph contains the Farey graph or $\TTT$ as a minor.
These two graphs are unique with this property up to minor-equivalence.
\end{abstract}


\makeatletter
\@namedef{subjclassname@2020}{%
  \textup{2020} Mathematics Subject Classification}
\makeatother

\vspace*{-1.9cm}
\maketitle

\vspace*{-1cm}
\noindent\begin{figure}[h]
\centering
\begin{minipage}{.45\textwidth}
  \centering
  \includegraphics[height=.7944\textwidth]{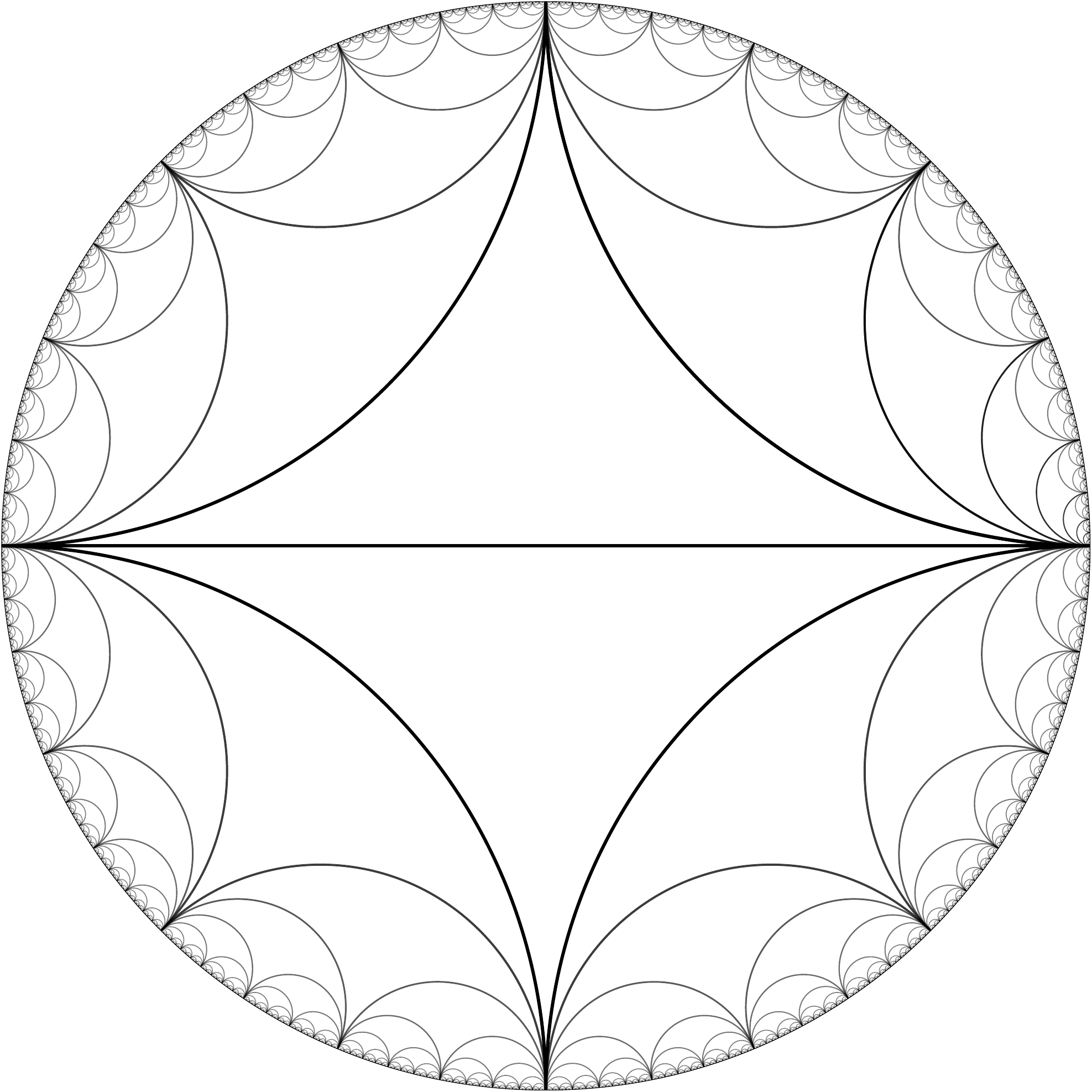}
  \captionof{figure}{The Farey graph}
  \label{fig:FareyGraph}
\end{minipage}%
\begin{minipage}{.55\textwidth}
  \centering
  \includegraphics[height=.65\textwidth]{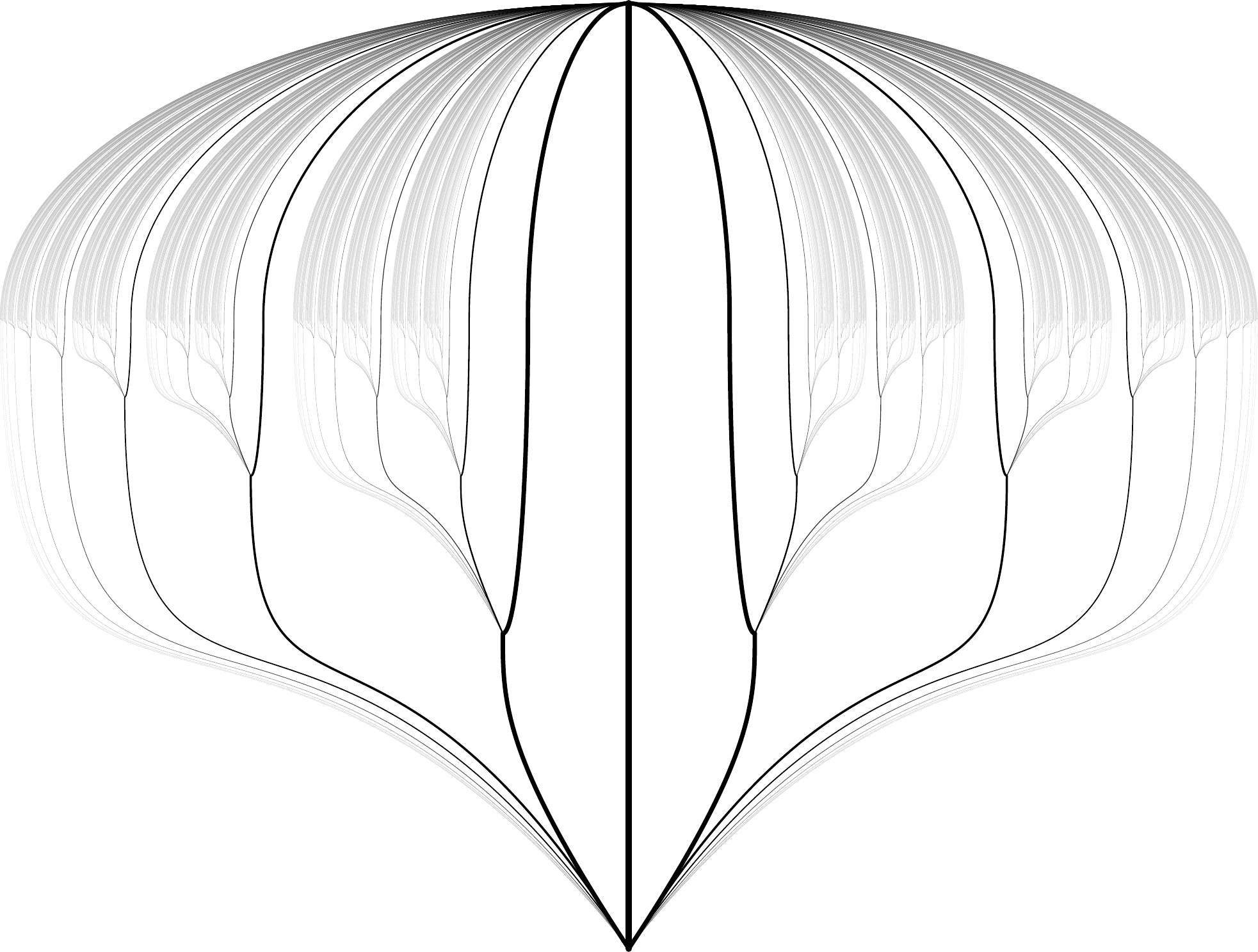}
  \captionof{figure}{The graph $\TTT$}
  \label{fig:TTT}
\end{minipage}
\end{figure}

\section{Introduction}

\noindent The Farey graph, shown in Figure~\ref{fig:FareyGraph} and surveyed in \cites{OfficeHoursGroupTheory,hatcher2017topology}, plays a role in a number of mathematical fields ranging from group theory and number theory to geometry and dynamics~\cite{OfficeHoursGroupTheory}.
Curiously, graph theory is not among these. 
In this paper we show that the Farey graph plays a central role in graph theory too: it is one of two infinitely edge-connected graphs that must occur as a minor in every infinitely edge-connected graph. 
Previously it was not known that there was any set of graphs determining infinite edge-connectivity by forming a minor-minimal list in this way, let alone a finite set.

Ramsey theory and the study of connectivity intersect in the problem of finding for any given connectivity $k$ a small set of $k$-connected subgraphs that occur in every $k$-connected graph, and thereby characterise $k$-connectedness.
To keep these unavoidable sets small for $k\ge 3$, the subgraph relation referred to above is usually relaxed to the graph minor relation. 
Here, a graph is a \emph{minor} of a graph $G$ if it can be obtained from a subgraph of $G$ by contracting connected (possibly infinite) induced disjoint subgraphs~\cite{DiestelBook5}.
We refer to~\cite{DiestelBook5}*{§9.4} or the introduction of~\cite{GollinHeuerKcon} for surveys on the known results for this problem and its variations~\cites{DiestelBook5,GenGridTheorem,GollinHeuerKcon,halin78,JoerisPhD,OporowskiOxleyThomas}.
Such sets of minor-minimal $k$-connected graphs are known only for $k\le 4$, and only for finite graphs~\cite{OporowskiOxleyThomas}.
These results of Oporowski, Oxley and Thomas were generalised to $k>4$ by Geelen and Joeris~\cite{GenGridTheorem} for finite graphs, and by Gollin and Heuer~\cite{GollinHeuerKcon} for infinite graphs, but with a different notion of connectivity.

For infinite connectivity, the problem asks for a small selection of infinitely connected graphs such that every infinitely connected graph contains at least one of the selected graphs as a minor.
Here, `infinitely connected' can be understood in two ways.
When it is understood as `infinitely vertex-connected', the answer is already known:
Every infinitely connected graph contains the countably infinite complete graph $K^{\aleph_0}$ as a minor~\cite{DiestelBook5}*{§8.1}.
But when `infinitely connected' is understood as `infinitely edge-connected' then, as we shall see, $K^{\aleph_0}$ is not the answer, and in fact no answer has been known.
Indeed it is not even clear a priori that there exists a finite set of unavoidable \iecon\ minors.
Any such unavoidable \iecon\ minors will be countable, because in every \iecon\ graph we can greedily find a countable \iecon\ subgraph.
But the countable graphs are not known to be well-quasi-ordered by the minor-relation.
It is therefore not clear that any minor-minimal set of \iecon\ graphs must be finite, nor even that such a minimal set exists.

In this paper we find a pair of \iecon\ graphs that occur unavoidably as minors in any \iecon\ graph, and which are unique with this property up to minor-equivalence: the Farey graph $F$, and the graph $\TTT$ obtained from the infinitely-branching tree $T_{\aleph_0}$ by joining an additional vertex $t$ to all its vertices (Figure~\ref{fig:TTT}).

\begin{mainresult}\label{mainresult}
Every \iecon\ graph contains either the Farey graph or $\TTT$ as a minor.
\end{mainresult}

\noindent The uniqueness of the pair $\{\mkern+.5\thinmuskip F,\,\TTT\,\}$, up to minor-equivalence, follows from the fact that they are not minors of each other (Lemmas~\ref{FareyGraphNotMinorOfTTT} and~\ref{TTTnotMinorOfFarey}):

\begin{mainresult}\label{mainresultTwo}
Let $M$ be any set of \iecon\ graphs such that every \iecon\ graph has a minor in~$M$ and no element of $M$ is a minor of another.
Then $M$ consists of two graphs, of which one is minor-equivalent to the Farey graph and the other is minor-equivalent to~$\TTT$.
\end{mainresult}

Theorem~\ref{mainresult} is best possible also in the sense that one cannot replace `minor' with `topological minor' in its wording (Theorem~\ref{mainresultBestPossible}).

Since both the Farey graph and $\TTT$ are planar, our result implies that every \iecon\ graph contains a planar \iecon\ graph as a minor.
Thus, in this sense, infinite edge-connectivity is an inherently planar property.

This paper is organised as follows. 
Section~\ref{sec:Preliminaries} formally introduces the Farey graph.
In Section~\ref{sec:Antichain} we show that the Farey graph and $\TTT$ are not minors of each other, and deduce Theorem~\ref{mainresultTwo}.
Theorem~\ref{mainresultBestPossible} above is proved there as well.
We outline the overall strategy of the proof of Theorem~\ref{mainresult} in Section~\ref{sec:Proofstrategy}.
The proof itself consists of two halves. The first half of the proof is carried out in Section~\ref{sec:HalfOne}, and the second half is carried out in Section~\ref{sec:HalfTwo}.
Section~\ref{sec:Outlook} gives an outlook, and Section~\ref{sec:appendix} contains the appendix.

\begin{ack}
I am grateful to Konstantinos Sta\-vro\-pou\-los for stimulating conversations.
\end{ack}

\section{Preliminaries}\label{sec:Preliminaries}

\noindent We use the notation of Diestel's book~\cite{DiestelBook5}. 
Two graphs are \emph{minor-equivalent} if they are minors of each other.
If $G$ is any graph and $X\subset V(G)$ is any vertex set, then we denote by $\partial X=\partial_G X$ the subset of $X$ formed by the vertices in $X$ that send an edge in $G$ to $V(G)\setminus X$.

\subsection{The Farey graph}
\input{FareyGraphDef2}

\subsection{Separation systems and \texorpdfstring{$\boldsymbol{S}$}{S}-trees}

Separation systems and $S$-trees are two fundamental tools in graph minor theory.
In this section we briefly introduce the definitions from~\cites{AbstractSepSys,DiestelBook5,RhdTreeSets} that we need.

A \emph{separation of a set} $V$ is an unordered pair $\{A,B\}$ such that $A\cup B=V$.
The ordered pairs $(A,B)$ and $(B,A)$ are its \emph{orientations}.
Then the \emph{oriented separations} of $V$ are the orientations of its separations.
The map that sends every oriented separation $(A,B)$ to its \emph{inverse} $(B,A)$ is an involution that reverses the partial ordering
\begin{align*}
    (A,B)\le (C,D)\;:\Leftrightarrow\;A\subset C\text{ and }B\supset D
\end{align*}
since $(A,B)\le (C,D)$ is equivalent to $(D,C)\le (B,A)$.

More generally, a \emph{separation system} is a triple $(\vS,{\le},{}^\ast)$ where $(\vS,{\le})$ is a partially ordered set and ${}^\ast\colon\vS\to\vS$ is an order-reversing involution.
We refer to the elements of $\vS$ as \emph{oriented separations}.
If an oriented separation is denoted by $\vs$, then we denote its \emph{inverse} $\vs^\ast$ as $\sv$, and vice versa.
That ${}^\ast$ is \emph{order-reversing} means $\vr\le\vs\leftrightarrow\rv\ge\sv$ for all $\vr,\vs\in\vS$.

A \emph{separation} is an unordered pair of the form $\{\vs,\sv\}$, and then denoted by $s$.
Its elements $\vs$ and $\sv$ are the \emph{orientations} of $s$.
The set of all separations $\{\vs,\sv\}\subset\vS$ is denoted by $S$.
When a separation is introduced as $s$ without specifying its elements first, we use $\vs$ and $\sv$ (arbitrarily) to refer to these elements.
Every subset $S'\subset S$ defines a separation system $\vSd:=\bigcup S'\subset\vS$ with the ordering and involution induced by~$\vS$.

Separations of sets, and their orientations, are an instance of this abstract setup if we identify $\{A,B\}$ with $\{\,(A,B)\,,(B,A)\,\}$.
Here is another example:
The set $\vE(T):=\{\,(x,y)\mid xy\in E(T)\,\}$ of all \emph{orientations} $(x,y)$ of the edges $xy=\{x,y\}$ of a tree $T$ forms a separation system with the involution $(x,y)\mapsto (y,x)$ and the natural partial ordering on $\vE(T)$ in which $(x,y)<(u,v)$ if and only if $xy\neq uv$ and the unique $\{x,y\}$--$\{u,v\}$ path in $T$ is $\mathring{x}yTu\mathring{v}=yTu$. 

In the context of a given separation system $(\vS,{\le},{}^\ast)$, a \emph{star (of separations)} is a subset $\sigma\subset\vS$ such that $\vr\le\sv$ for all distinct $\vr,\vs\in\sigma$; see~\cite{DiestelBook5}*{Fig.~12.5.1} for an illustration.\footnote{Officially, in~\cite{AbstractSepSys} a star $\sigma$ is additionally required to consist only of oriented separations $\vs$ satisfying $\vs\neq\sv$. In this paper, however, all separations considered will satisfy this condition, which is why we drop it here.}
If $t$ is a node of a tree $T$, then the set
\begin{align*}
    \vF_{\! t}:=\{\,(x,t)\mid xt\in E(T)\,\}
\end{align*}
is a star in $\vE(T)$.

An $S$-\emph{tree} is a pair $(T,\alpha)$ such that $T$ is a tree and $\alpha\colon\vE(T)\to\vS$ propagates the ordering on $\vE(T)$ and commutes with inversion: that $\alpha(\ve)\le\alpha(\vf)$ if $\ve\le\vf\in\vE(T)$ and $(\alpha(\ev))^\ast=\alpha(\ve)$ for all $\ve\in \vE(T)$; see \cite{DiestelBook5}*{Fig.~12.5.2} for an illustration.
Thus, every node $t\in T$ is associated with a star $\vF_{\! t}$ in $\vE(T)$ which $\alpha$ sends to a star $\alpha[\vF_{\! t}]$ in~$\vS$.
A tree-decomposition $(T,\cV)$, for example, makes $T$ into an $S$-tree for the set of separations it induces~\cite{DiestelBook5}*{§12.5}.
For oriented edges $(x,y)\in\vE(T)$ we will write $\alpha(x,y)$ instead of $\alpha((x,y))$.
Note that $S$-trees are `closed under taking minors' in the sense that if $(T,\alpha)$ is an $S$-tree and $T'\minor T$, then $(\,T',\,\alpha\rest \vE(T')\,)$ is again an $S$-tree when we view $E(T')$ as a subset of $E(T)$.

\section{Uniqueness and topological minors}\label{sec:Antichain}

\subsection{Uniqueness}

\noindent In this section we show that the pair $\{\mkern+.5\thinmuskip F,\,\TTT\,\}$ is unique up to minor-equivalence:

\begin{customthm}{\ref{mainresultTwo}}
Let $\cH$ be any set of \iecon\ graphs such that every \iecon\ graph has a minor in~$\cH$ and no element of $\cH$ is a minor of another.
Then $\cH$ consists of two graphs, of which one is minor-equivalent to the Farey graph and the other is minor-equivalent to~$\TTT$.
\end{customthm}

\noindent This will follow easily from the following two lemmas:

\begin{lemma}\label{FareyGraphNotMinorOfTTT}
The Farey graph is not a minor of $\TTT$.
\end{lemma}
\begin{proof}
The Farey graph contains two disjoint cycles, but $\TTT$ does not.
\end{proof}

\begin{lemma}\label{TTTnotMinorOfFarey}
The graph $\TTT$ is not a minor of the Farey graph.
\end{lemma}

\begin{proof}[Proof of Theorem~\ref{mainresultTwo}]
We write $\cG=\{\mkern+.5\thinmuskip F,\,\TTT\,\}$ and note that neither element of $\cG$ is a minor of another by Lemmas~\ref{FareyGraphNotMinorOfTTT} and~\ref{TTTnotMinorOfFarey}.
Every graph $H\in\cH$ contains a graph $G\in\cG$ as a minor (Theorem~\ref{mainresult}) which in turn contains a graph $H'\in\cH$ as a minor, and then $H\rminor G\rminor H'$ implies $H=H'$ because no element of $\cH$ is a minor of another.
Thus, every graph in $\cH$ is minor-equivalent to some graph in $\cG$ and, conversely, every graph in $\cG$ is minor-equivalent to some graph in $\cH$ by symmetry.
Since no two graphs in $\cH$ or in $\cG$ are comparable with regard to the minor-relation, we deduce that minor-equivalence induces a bijection between $\cH$ and $\cG$.
\end{proof}

Showing that $\TTT$ is not a minor of the Farey graph requires more effort, and some preparation.
A \emph{comb} is the union of a ray $R$ (the comb's \emph{spine}) with infinitely many disjoint finite paths, possibly trivial, that have precisely their first vertex on~$R$. 
The last vertices of those paths are the \emph{teeth} of this comb.
Given a vertex set $U$, a \emph{comb attached to} $U$ is a comb with all its teeth in $U$, and a \emph{star attached to} $U$ is a subdivided infinite star with all its leaves in $U$.
The following lemma is~\cite{DiestelBook5}*{Lemma~8.2.2}, see also the series~\cites{StarComb1StarsAndCombs,StarComb2TheDominatedComb,StarComb3TheUndominatedComb,StarComb4TheUndominatingStar}.

\begin{lemma}[Star-comb lemma]
Let $U$ be an infinite set of vertices in a connected graph $G$.
Then $G$ contains either a comb \at $U$ or a star \at $U$.
\end{lemma}

\begin{proof}[Proof of Lemma~\ref{TTTnotMinorOfFarey}]
Since $K_{2,\aleph_0}$ is a subgraph of $\TTT$, it suffices to show that the Farey graph does not contain $K_{2,\aleph_0}$ as a minor.
So let us assume for a contradiction that the Farey graph contains a $K_{2,\aleph_0}$ minor.
By applying the star-comb lemma inside the branch sets of the two infinite-degree vertices of $K_{2,\aleph_0}$ if necessary, and using that the Farey graph does not contain infinitely many independent paths between any two of its vertices, we find that our model of $K_{2,\aleph_0}$ contains a subdivision $G$ of one of the following two graphs $G_1$ and $G_2$.
The graph $G_1$ is the ladder with every rung subdivided exactly once, i.e., it is the disjoint union of two rays $R=v_1v_2\ldots$ and $R'=v_1'v_2'\ldots$ with infinitely many disjoint $R$--$R'$ paths~$v_n z_n v_n'$ ($n\in\N$).
And the graph $G_2$ is obtained from $G_1$ by contracting $R'$ to a single vertex that we call~$d$.

In either case, the sole end of $G\subset F$ is included in a unique end $\omega$ of $F$.
The end $\omega$ chooses, for every $n\in\N$, a blue edge $e_n\in F_n$ with vertex set $X_n$ for which it lives in the component $C_n$ of $F-X_n$ avoiding $F_n$.
Then $C_n$ has neighbourhood $X_n$, and so does the other component $D_n$ of $F-X_n$.
We remark that, by the construction of the Farey graph, for every vertex of $F$ there is a number $n$ such that the vertex is not contained in $C_n$.
For all $n$ the two vertex sets $X_n$ and $X_{n+1}$ together induce a triangle $\Delta_n$ in $F$.
We write $x_n$ for the vertex in which $X_n$ and $X_{n+1}$ meet, and we write $Y_n$ for vertex set consisting of the other two vertices of the triangle $\Delta_n$.
The graph $F-\Delta_n$ has precisely three components, namely $D_n$ and $C_{n+1}$ and a third component with neighbourhood $Y_n$ which we denote by $H_n$.

First, we consider the case that $G\subset F$ is a subdivision of~$G_1$, and we write $\hat{R}$ and $\hat{R}'$ for the subdivisions of $R$ and $R'$ in~$G$.
Then there cannot be a number $N$ such that $x_n=x_N$ for all $n\ge N$:
Indeed, for every $k\in\N$ there is a number $f(k)\ge k$ such that both $v_k$ and $v_k'$ are not contained in $C_{f(k)}$ and, as a consequence, $x_{f(k)}$ must be contained in $v_k \hat{R}\cup v_k'\hat{R}'$.
Thus, every vertex of $F$ lies in a component $D_n$ eventually (and $\omega$ is undominated).
Let $N$ be the least number for which the first vertices of $\hat{R}$ and $\hat{R}'$ lie in $D_N$.
To derive a contradiction from $G\subset F$, let us consider any $\hat{R}$--$\hat{R}'$ path $P\subset G$ that lies entirely in the component $C_N$, and consider the maximal number $n$ for which $P$ avoids $D_n$, noting $n\ge N$.
Since the two rays $\hat{R}$ and $\hat{R}'$ induce a bipartition of the 2-set $X_{n+1}$, the path $P$ cannot meet $C_{n+1}$ without contradicting the maximality of~$n$.
Therefore, the path $P$ is contained entirely in $F[H_n\sqcup\Delta_n]$.
Without loss of generality we have $x_n\in \hat{R}$.
Then $Y_n\subset \hat{R}'$ follows.
But now the non-empty subpath $\mathring{P}$ must be contained in $H_n$, contradicting that $H_n$ has neighbourhood $Y_n\subset \hat{R}'$.

Second, we consider the case that $G\subset F$ is a subdivision of $G_2$, and again we write $\hat{R}$ for the subdivision of $R$ in $G$.
Since $d\in G_2$ dominates the end of $G_2$, the end $\omega$ is dominated in $F$ by $d$.
Let $N$ be the least number such that both $d$ and the first vertex of $\hat{R}$ are not contained in~$C_N$.
Then $d=x_N=x_n$ for all $n\ge N$ because $d$ dominates $\omega$.
Thus, $Y_n\subset \hat{R}$ for all $n\ge N$.
Now consider any $d$--$\hat{R}$ path $P\subset G$ with $\mathring{d}P\subset C_N$ and choose $n$ maximal with the property that the non-empty subpath $\mathring{P}$ avoids $D_n$, noting $n\ge N$.
Then $\mathring{P}\subset H_n$ follows because of $Y_n\subset \hat{R}$, contradicting that $d=x_n$ does not lie in the neighbourhood $Y_n$ of $H_n$.
\end{proof}

\subsection{Minor versus topological minor}\label{sec:mainresultBestPossible}

\noindent Theorem~\ref{mainresult} is best possible in the sense that one cannot replace `minor' with `topological minor' in its wording:

\begin{theorem}\label{mainresultBestPossible}
There exists an \iecon\ graph that contains neither the Farey graph nor $\TTT$ as a topological minor.
\end{theorem}
\begin{proof}
By a recent result~\cite{OrderCompatiblePaths} there exists an \iecon\ graph $G$ which does not contain infinitely many edge-disjoint pairwise order-compatible paths between any two of its vertices.
Here, two \mbox{$u$--$v$} paths are \emph{order-compatible} if they traverse their common vertices in the same order.
Then the graph $G$ does not contain a subdivision of the Farey graph or of $\TTT$ because both the Farey graph and $\TTT$ have pairs of vertices with infinitely many edge-disjoint pairwise order-compatible paths between them.
\end{proof}

\section{Overall proof strategy}\label{sec:Proofstrategy}

\noindent Our aim for the remainder of this paper is to show that every \iecon\ graph contains either the Farey graph or $\TTT$ as a minor (Theorem~\ref{mainresult}).
The proof consists of two halves.
In the first half (Section~\ref{sec:HalfOne}) we show that every \iecon\ graph without a $\TTT$ minor is `robust' (Theorem~\ref{PowerHorse}), explained below.
Then, in the second half (Section~\ref{sec:HalfTwo}), we employ Theorem~\ref{PowerHorse} to prove that every \iecon\ graph without a $\TTT$ minor must contain a Farey graph minor, completing the proof of Theorem~\ref{mainresult}.

The Farey graph and $\TTT$ are both \iecon , but in different ways.
The infinite edge-connectivity of the Farey graph, on the one hand, is robust in that deleting the two endvertices of an edge always leaves only \iecon\ components.
The infinite edge-connectivity of $\TTT$, on the other hand, is fragile in that deleting $t$ results in a tree.
In the first half of the proof of Theorem~\ref{mainresult} we show that every \iecon\ graph without a $\TTT$ minor is essentially robust, not fragile (Theorem~\ref{PowerHorse}).

In the second half of the proof of Theorem~\ref{mainresult} we construct a Farey graph minor in an arbitrary \iecon\ $\TTT$ free graph~$G$.
By Lemma~\ref{FareyGisMinorOfHalvedFareyG} it suffices to construct a halved Farey graph minor.
Using that $G$ is robust by Theorem~\ref{PowerHorse}, we shall essentially prove the following assertion:

\emph{For every two vertices $u$ and $v$ of $G$ there exist two induced subgraphs $H_u,H_v\subset G$ containing $u$ and $v$ respectively and which satisfy the following conditions:
\begin{enumerate}
\item $X:=V(H_u)\cap V(H_v)$ is finite, non-empty and connected in $G$;
\item both $H_u/X$ and $H_v/X$ are \iecon ;
\item $X$ avoids $u$ and $v$;
\item $uX$ is an edge of $H_u/X$ and $vX$ is an edge of $H_v/X$.
\end{enumerate}}
\noindent If we choose $u$ and $v$ to form an edge of $G$, then the three vertices $u,v$ and $X$ span a triangle $\hFG_1$ in $(H_u\cup H_v)/X$.
And since $H_u/X$ and $H_v/X$ are both \iecon\ and robust again, we can reapply the assertion in $(H_u/X)-uX$ to $u$ and $X$, and in $(H_v/X)-vX$ to $v$ and $X$.
By iterating this process, we obtain a halved Farey graph minor in the original graph $G$ at the limit, and this will complete the proof of Theorem~\ref{mainresult}.

\section{Robustness}\label{sec:HalfOne}

\noindent The aim of this section is to prove Theorem~\ref{PowerHorse} which has been outlined in the previous section.
Our proof proceeds in three steps.
First, we provide some tools that will help us to (i) identify \iecon\ `parts' of arbitrary graphs and (ii) allow us to distinguish all these `parts' at once in a tree-like way.
In the second step, then we employ these tools to analyse the components of $G-u-v$ for \iecon\ graphs $G$ and vertices $u,v$ of~$G$.
In the third step, we proceed to prove Theorem~\ref{PowerHorse}.

\subsection{Finitely separating spanning trees}

Let $G$ be any graph.
Two vertices of $G$ are said to be \emph{finitely separable} in $G$ if there is a finite set of edges of $G$ separating them in~$G$.
If every two distinct vertices of $G$ are finitely separable, then $G$ itself is said to be \emph{finitely separable}.
An equivalence relation ${\sim}={\sim}_G$ is declared on the vertex set of $G$ by letting $x\sim y$ whenever $x$ and $y$ are not finitely separable.
The graph $\tilde{G}$ is defined on $V(G)/{\sim}$ by declaring $XY$ an edge whenever $X\neq Y$ and there is an $X$--$Y$ edge in~$G$.
Note that the graph $\tilde{G}$ is always finitely separable.
A~spanning tree $T$ of $G$ is \emph{finitely separating} if all its fundamental cuts are finite.
By standard arguments of topological infinite graph theory, the following theorem is equivalent to Theorem~6.3 in~\cite{duality}.
See the appendix in Section~\ref{sec:appendix} for the arguments.

\begin{theorem}\label{finSepSTsExist}
Every connected finitely separable graph has a finitely separating spanning tree.
\end{theorem}

Usually, we will employ Theorem~\ref{finSepSTsExist} to find a finitely separating spanning tree $T$ of $\tilde{G}$ that we will then use to analyse the overall structure of $G$ with regard to infinite edge-connectivity.
In this context, the nodes of $T\subset\tilde{G}$ will also be viewed as the vertex sets of $G$ that they formally are.
When we view a node of $T$ as a vertex set of $G$ we will refer to it as \emph{part} for clarity.

Every finitely separating spanning tree $T\subset\tilde{G}$ defines an $S$-tree $(T,\alpha)$ for the set $S=\fcut(G)$ of all the separations of the vertex set $V(G)$ that are bipartitions induced by finite bonds of~$G$:
Let the map $\alpha$ send every oriented edge $(t_1,t_2)\in\vE(T)$ to the ordered pair $(\,\bigcup V(T_1)\,,\,\bigcup V(T_2)\,)$ for the two components $T_1$ and $T_2$ of $T-t_1t_2$ containing $t_1$ and $t_2$ respectively.
Then $\alpha(t_1,t_2)$ clearly is an oriented bipartition of $V(G)$.
Moreover, we have $\alpha(\ve)\le\alpha(\vf)$ whenever $\ve\le\vf\in\vE(T)$ and $(\alpha(\ev))^\ast=\alpha(\ve)$ for all $\ve\in\vE(T)$.
It remains to show that $\alpha(\ve)$ always stems from a finite bond of~$G$.
%
%
%
%
%
%
%
For this, it suffices to show that if $\{A,B\}\in\fcut(\tilde{G})$ then $\{\,\bigcup A\,,\,\bigcup B\,\}\in\fcut(G)$, because all the fundamental cuts of $T$ are finite bonds.
Between every two $\sim$-classes $U$ and $W$ of $G$ there are only finitely many edges, because $u\in U$ is separated from $w\in W$ by a finite cut of $G$ and then $U$ and $W$ must respect this finite cut.
Hence the finitely many $A$--$B$ edges in $\tilde{G}$ give rise to only finitely many $(\bigcup A)$--$(\bigcup B)$ edges in $G$, and these are all $(\bigcup A)$--$(\bigcup B)$ edges in $G$.
Using that $G$ contains for all ${\sim}$-equivalent vertices $x$ and~$y$ an $x$--$y$ path avoiding the finitely many $(\bigcup A)$--$(\bigcup B)$ edges, it is straightforward to show that both $G[\,\bigcup A\,]$ and $G[\,\bigcup B\,]$ are connected.


The \emph{part} of a star $\{\,(A_i,B_i)\mid i\in I\,\}$ of separations of a given set is the intersection $\bigcap_{i\in I} B_i$.
If~$(T,\alpha)$ is a $\fcut(G)$-tree that is defined by a finitely separating spanning tree $T$ of~$\tilde{G}$, then for every node $t\in T$ the part of the star $\alpha [\vF_{\! t}]\subset\vfcut(G)$ associated with $t$ is equal to the part $t\subset V(G)$.
And the parts $t\subset V(G)$ in turn are precisely the $\sim$-classes of~$G$.
Thus, in this sense, by Theorem~\ref{finSepSTsExist} every connected graph admits a tree structure that displays all its $\sim$-classes.

Parts of infinite stars in $\vfcut(G)$ can be made connected for a reasonable price:

\begin{lemma}\label{Taleph0BranchSetConstructionFragment}
Suppose that $G$ is a connected graph, that $\sigma=\{\,(A_i,B_i)\mid i\in I\,\}$ is an infinite star in $\vfcut(G)$ 
and that $i_\ast\in I$ is given.
Then there is an infinite subset $J\subset I$ containing $i_\ast$ such that the part of the infinite substar $\{\,(A_j,B_j)\mid j\in J\,\}\subset\sigma$ is connected in $G$.
\end{lemma}

\begin{proof}
For each $i\in I$ we write $F_i$ for the finite bond $E(A_i,B_i)$ of $G$.

Inductively, we construct an ascending sequence $T_0\subset T_1\subset\cdots$ of finite trees in $G$ 
together with a sequence of distinct indices $i_0,i_1,\dots$ in $I\setminus\{i_\ast\}$ such that, for all $n\in\N$ and $J_n:=\{i_\ast\}\sqcup\{i_0,\ldots,i_{n-1}\}$, the tree $T_n$ is a subgraph of $G_n:=G[\,\bigcap_{j\in J_n}B_j\,]$ containing all $\partial B_j$ with $j\in J_n$.
Then the tree $T:=\bigcup_{n\in\N}T_n$ will ensure that $G_\infty:=G[\,\bigcap_{j\in J}B_j\,]$ is connected for $J:=\bigcup_{n\in\N}J_n$.
(For whenever a path in $G$ connecting two given vertices in $G_\infty$ uses vertices that are not in $G_\infty$, then the path crosses one of the bonds $F_j$, and the number of bonds crossed can be decreased by replacing path segments with detours in $T\supset \partial B_j$ because $T\subset G_\infty$.
Therefore, choosing a path that crosses as few bonds $F_j$ as possible will suffice to find a path that lies entirely in $G_\infty$.)

To start the construction, let $T_0$ be any finite tree in $G[B_{i_\ast}]$ that contains $\partial B_{i_\ast}$.
At step $n+1$ of the construction, suppose that we have already constructed $T_n$ and $J_n$.
As $T_n$ is finite, we find an index $i_n\in I\setminus J_n$ for which $A_{i_n}$ avoids $T_n$, ensuring $T_n\subset G_{n+1}$.
To ensure that $T_n$ can be extended in $G_{n+1}$ to a finite tree $T_{n+1}$ that contains $\partial B_{i_n}$, it suffices to show that $G_{n+1}$ is connected.
Given any two vertices in $G_{n+1}$, consider any path between them in $G[B_{i_n}]$, chosen to cross as few of the finite bonds $F_j$ with $j\in J_n$ as possible.
Then the path avoids all these $F_j$, for otherwise the number of bonds crossed could be decreased by replacing path segments with detours in $T_n\supset\bigcup_{j\in J_n}\partial B_j$.
Therefore, $G_{n+1}$ is connected.
\end{proof}

\subsection{Analysing the components}

Now we analyse the components of $G-u-v$ for \iecon\ graphs $G$ and vertices $u,v$ of~$G$.
The main results here are the two Lemmas~\ref{fcutTreeContainsBinaryGivesTTT} and~\ref{fcutTreeArrowBarrage}.
Here is the first main lemma:

\begin{lemma}\label{fcutTreeContainsBinaryGivesTTT}
Suppose that $G$ is an \iecon\ graph, that $u,v$ are two distinct vertices of $G$, and that $C$ is a component of $G-u-v$.
If $\tilde{C}$ has a finitely separating spanning tree that contains a subdivision of the infinite binary tree, then $G[C+u+v]$ contains $\TTT$ as a minor.
\end{lemma}
\begin{proof}
Consider any finitely separating spanning tree of $\tilde{C}$ that contains a subdivision of the infinite binary tree.
Then this spanning tree also contains $T_{\aleph_0}$ as a contraction minor which gives rise to a $\fcut(C)$-tree $(T,\alpha)$. 
Next, we fix any root $r\in T$, and for every edge $e\in T$ we fix $\ve$ as its orientation pointing away from the root~$r$ (the orientation $\ve=(x,y)$ of $e=\{x,y\}$ satisfying $x\in rTy$).
Let $O:=\{\,\ve\mid e\in E(T)\,\}$.
Since $G$ is \iecon , $O$ is equal to the union $O_u\cup O_v$ where $\ve\in O_w$ (for $w=u,v$) if and only if $w$ sends an edge in $G$ to $B$ for $\alpha(\ve)=(A,B)$.
Now $O_u$ is cofinal\footnote{A subset $X$ of a poset $P=(P,{\le})$ is \emph{cofinal} in $P$, and ${\le}$, if for every $x\in X$ there is a $p\in P$ with $p\ge x$.} in $O\subset\vE(T)$ or there is an oriented edge $\ve\in O$ with $O_v$ cofinal in $\uc{\ve}_O:=\{\,\vf\in O\mid\ve\le\vf\,\}$.
In either case, there is $\ve\in O$ with $O_u$ or $O_v$ cofinal in~$\uc{\ve}_O$.
Without loss of generality $O_u$ is cofinal in $\uc{\ve}_O$ for some $\ve\in O$.
By replacing $T$ with one of its subtrees and restricting $\alpha$ accordingly, we may even assume that $O_u$ is cofinal in~$O$.
In fact, then $O_u=O$ follows as $O_u$ is down-closed in~$O$.
We will use this to show $\TTT\minor G[C+u]$.

For this, we enumerate the vertices of $T_{\aleph_0}$ as $x_0,x_1,\ldots$ such that every $x_n$ is neighbour to some earlier $x^k$ ($k<n$).
Inductively, we construct a sequence $W_0,W_1,\ldots$ of disjoint connected vertex sets $W_n\subset V(C)$, a sequence $w_0,w_1,\ldots$ of vertices $w_n\in W_n$, and a sequence $t_0,t_1,\ldots$ of distinct nodes $t_n\in T$ such that, for all $n\in\N$: 
\begin{enumerate}
    \item $uw_n\in G$;
    \item $C$ contains a $W_i$--$W_j$ edge ($i,j\le n$) whenever $x_ix_j\in T_{\aleph_0}$;
    \item $w_n$ is contained in the part of the star $\alpha[\vF_{\! t_n}]$;
    \item for all $k\le n$ there are infinitely many oriented edges $\ve\in O\cap (\vF_{\! t_k})^\ast$ such that, for $\alpha(\ve)=(B,A)$, the vertex set $W_k$ contains $\partial_C B$ while $A$ is avoided by all $W_i$ with $i\le n$.
\end{enumerate}
Once the construction is completed, the sets $W_n$ and $\{u\}$ will give rise to a model of $\TTT$ in $G[C+u]$ by~(i) and~(ii).

At the construction start, we choose any neighbour $w_0$ of $u$ in $C$ (which exists as $O_u=O$ and $T$ is infinite), guaranteeing~(i).
Then $t_0$ is defined by~(iii).
Applying Lemma~\ref{Taleph0BranchSetConstructionFragment} in $C$ to the infinite star $\alpha[\vF_{\! t_0}]$ yields an infinite substar whose connected part $W_0\subset V(C)$ contains $w_0$ and satisfies both (ii) and~(iv) trivially.

At step $n>0$ of the construction, consider the $k<n$ for which $x_k x_n$ is an edge of $T_{\aleph_0}$, and pick an edge $\ve\in O\cap (\vF_{\! t_k})^\ast$ that (iv) provides for $k\le n-1$.
If we write $\alpha(\ve)=(B,A)$, then the vertex set $W_k$ contains $\partial_C B$ while $A$ is avoided by all $W_i$ with $i\le n-1$.
Using $O_u=O$ we find a neighbour $w_n$ of $u$ in~$A$ giving~(i), and $w_n$ defines $t_n$ by~(iii).
Then we apply Lemma~\ref{Taleph0BranchSetConstructionFragment} in $C$ to the infinite star
\begin{align*}
    \{\,(A_i,B_i)\mid i\in I\,\}:=\alpha\big[\,(\vF_{\! t_n}\!\setminus O)\cup\{\ve\}\,\big]
\end{align*}
where we take $i_\ast\in I$ to be the index of the separation~$\alpha(\ve)$.
This yields an infinite substar whose connected part $W_n\subset V(C)$ contains $w_n$ and satisfies~(ii) because $W_n$ contains $\partial_C A$ while $W_k$ contains $\partial_C B$.
Using the infinite substar and the choice of $\ve$ it is straightforward to verify~(iv) for all $k\le n$.
\end{proof}

Our second main lemma, Lemma~\ref{fcutTreeArrowBarrage}, requires some preparation.

\begin{definition}[Arrow]
Suppose that $u$ and $v$ are two distinct vertices.

An \emph{arrow from} $u$ \emph{to} $v$ is a graph $G$ that arises from the two vertices $u$ and $v$ by disjointly adding an \iecon\ graph $H$, adding a $u$--$H$ edge $uh$, and adding infinitely many $v$--$(H-h)$ edges.
Then $H$ is the arrow's \emph{payload}, 
$u$ is its \emph{nock} and $v$ is its \emph{head}.

An \emph{arrow barrage from} $u$ \emph{to} $v$ is a countably infinite union $\bigcup_{n\in\N}A_n$ of arrows $A_n$ from $u$ to $v$ such that $A_n$ and $A_m$ do not meet in any vertices other than $u$ and~$v$ for all $n\neq m$.
Then $u$ and $v$ are the \emph{nock} and \emph{head} of the arrow barrage.

When we say that some graph contains an arrow (barrage) minor \emph{from} $x$ \emph{to} $y$ for two vertices $x$ and $y$, we mean that the graph contains an arrow (barrage) minor such that the branch set corresponding to the arrow (barrage)'s nock contains $x$ while the branch set corresponding to the arrow (barrage)'s head contains $y$.
\end{definition}

The next definition captures the concept of recursive pruning that Diestel describes in his book~\cite{DiestelBook5} as follows:

\begin{definition}[Recursive pruning]
Let $T$ be any tree, equipped with a root and the corresponding tree-order on its vertices.
We recursively label the vertices of $T$ by ordinals, as follows.
Given an ordinal $\alpha$, assume that we have decided for every $\beta<\alpha$ which of the vertices of $T$ to label $\beta$, and let $T_\alpha$ be the subgraph of $T$ induced by the vertices that are still unlabelled. Assign label $\alpha$ to every vertex $t$ of $T_\alpha$ whose up-closure $\uc{t}_{T_\alpha}=\uc{t}_T\cap T_\alpha$ in $T_\alpha$ is a chain.
The recursion terminates at the first $\alpha$ not used to label any vertex; for this $\alpha$ we put $T_\alpha=:T^\ast$.
We call $T$ \emph{recursively prunable} if every vertex of $T$ gets labelled in this way, i.e., if $T^\ast=\emptyset$.
\end{definition}

\begin{proposition}[{\cite{DiestelBook5}*{Proposition~8.5.1}}]\label{RecPrunableBinary}
A rooted tree is recursively prunable if and only if it contains no subdivision of the infinite binary tree.
\end{proposition}

The next lemma is an observation that we will use often:

\begin{lemma}\label{basicFact}
Suppose that $G$ is an \iecon\ graph, that $u,v$ are two distinct vertices of $G$, and that $C$ is a component of $G-u-v$.
If $T$ is a finitely separating spanning tree of $\tilde{C}$ and $t\in T$ has finite degree in $T$, then $C[t]$ is \iecon\ and either $u$ or $v$ sends infinitely many edges in $G$ to the part $t\subset V(C)$.
\end{lemma}
\begin{proof}
As $t$ has finite degree in $T$, the finite fundamental cuts of the edges of $T$ incident with $t$ together give rise to a finite cut of $C$ with the part $t$ as one of its sides.
Thus, in the graph $G$ every vertex in the part $t$ can send only finitely many edges to $C-t$, at most one edge to each of $u$ and $v$, and some edges to the rest of the part~$t$.
As every vertex of the \iecon\ graph $G$ has infinite degree, it follows that the part $t$ must be infinite.
And since no two vertices in $t$ are finitely separable in $C$ while $t$ is separated from the rest of $C$ by a single finite cut, it follows that $C[t]$ is \iecon .
Finally, at least one of $u$ and $v$ sends infinitely many edges to the part $t$, because otherwise $t$ is separated from the rest of $G$ by a finite cut, contradicting that $G$ is \iecon .
\end{proof}

Here is the second main lemma of this section:

\begin{lemma}\label{fcutTreeArrowBarrage}
Suppose that $G$ is an \iecon\ graph, that $u,v$ are two distinct vertices of $G$, and that $C$ is a component of $G-u-v$ such that $u$ sends at least one edge to $C$.
If $\tilde{C}$ has a recursively prunable finitely separating rooted spanning tree $T$ such that $u$ sends no edges to parts $t\in T$ that are finite-degree nodes of $T$, then $G[C+u+v]$ contains an arrow barrage minor from $u$ to $v$.
\end{lemma}
\begin{proof}
Given $T$, we let $X\subset V(T)$ consist of the $0$-labelled nodes of $T$ that are minimal in the tree-order.
Then the nodes in $X$ form a maximal antichain in the tree-order, giving $T=\uc{X}\cup\dc{X}$, as $T$ is recursively prunable.
Note that all the nodes in $\uc{X}$ have degree at most two in~$T$.
We claim that $X$ must be infinite.
Indeed, if $X$ is finite, then so is $\dc{X}$, and in particular $T$ is locally finite.
But then $u$ may send 
no edges to $C$  
by assumption, contradicting our other assumption that $u$ does send an edge to~$C$.
Therefore, $X$ must be infinite.

Recall that the finitely separating spanning tree $T\subset\tilde{C}$ gives rise to a $\fcut(C)$-tree $(T,\alpha)$.
For every $x\in X$ let us write $(A_x,B_x):=\alpha(x,p_x)$ for the predecessor $p_x$ of $x$ in $T$.
As $u$ sends some edges to $C$, but 
none to the parts in $\uc{X}$, there is a neighbour $w$ of $u$ in the part $\bigcap_{x\in X}B_x$ of the star $\sigma:=\{\,(A_x,B_x)\mid x\in X\,\}$.
By Lemma~\ref{Taleph0BranchSetConstructionFragment} we find an infinite subset $Y\subset X$ such that the part of the infinite substar $\sigma':=\{\,(A_y,B_y)\mid y\in Y\,\}\subset\sigma$ is connected.
Note that $w$ is contained in the part of $\sigma'$ because the part of $\sigma$ is included in the part of $\sigma'$.
We now find an arrow barrage minor from $u$ to $v$ in $G[C+u+v]$ as follows.
For the branch set of the nock we take the part of $\sigma'$ plus the vertex~$u$.
For the branch set of the head we take~$\{v\}$.
The payloads we let be modelled by the subgraphs~$C[y]$, one for every~$y\in Y$ (here, each $C[y]$ is \iecon\ and sends infinitely many edges in $G$ to $v$ by Lemma~\ref{basicFact} and $Y\subset X$).
\end{proof}

\subsection{Football minors}

We are almost ready now to prove Theorem~\ref{PowerHorse}.
But first, we prove an intermediate proposition, which requires the following lemma and definition:

\begin{lemma}\label{finiteEdgeContractionOK}
If $G$ is an \iecon\ graph and $G'$ is 
obtained from $G$ by contracting disjoint finite vertex sets that are possibly disconnected, then $G'$ is \iecon\ as well.
\end{lemma}
\begin{proof}
To show that $G'$ is \iecon , consider any two distinct vertices $x$ and $y$ of $G'$, and choose vertices $\check{x}\in x$ and $\check{y}\in y$ of $G$.
Now, in the \iecon\ graph $G$ we choose infinitely many pairwise edge-disjoint $\check{x}$--$\check{y}$ paths $P_0,P_1,\ldots$ as follows.
To get started, choose $P_0$ arbitrarily.
At step $n>0$, consider all the branch sets that are met by some $P_k$ with $k<n$, and let $X_n$ be their union.
Then $X_n$ is finite, and we let $P_n$ be an $\check{x}$--$\check{y}$ path in $G$ that avoids all the finitely many edges of $G$ running inside~$X_n$.

Now every $\check{x}$--$\check{y}$ path $P_n\subset G$ gives rise to some $x$--$y$ path $P_n'\subset G'$ satisfying $E(P_n')\subset E(P_n)$ by a slight abuse of notation.
We claim that the paths $P_0',P_1',\ldots$ are all edge-disjoint.
For this, consider any two paths $P_n'$ and $P_m'$ with $n<m$, and let $e$ be any edge of~$P_n'$.
Then~$e$, viewed as an edge of $G$, runs between two branch sets that $P_n$ meets because it uses~$e$.
Hence these two branch sets are both included in $X_m$, and so $P_m$ does not use any of the edges running between them.
In particular, $P_m'$ does not use $e$.
\end{proof}

\begin{definition}[Football, Muscle]
Suppose that $u$ and $v$ are two distinct vertices.

A \emph{football} with \emph{endvertices} $u$ and $v$ is an \iecon\ graph $G$ containing $u$ and $v$ such that $G-u-v$ is again \iecon .

When we say that some graph contains a football minor \emph{connecting} two vertices $x$ and $y$ we mean that the graph contains a football minor with some endvertices $u$ and $v$ such that the branch set corresponding to $u$ contains $x$ and the branch set corresponding to $v$ contains $y$ (or vice versa).

A \emph{muscle} with \emph{endvertices} $u$ and $v$ is a graph $G$ that is obtained from the vertices $u$ and $v$ by disjointly adding an \iecon\ graph $H$ and adding one $u$--$H$ edge $ux$ and one $v$--$H$ edge $vy$ such that $x\neq y$.

A \emph{muscle barrage} with \emph{endvertices} $u$ and $v$ is a countably infinite union $\bigcup_{n\in\N}G_n$ of muscles $G_n$ with endvertices $u$ and $v$ such that $G_n$ and $G_m$ do not meet in any vertices other than $u$ and $v$ for all $n\neq m$.

Muscle (barrage) minors \emph{connecting} two vertices are defined like for footballs.
\end{definition}

\begin{proposition}\label{FindingAnArrowBarrage}
Suppose that $G$ is an \iecon\ graph, that $u,v$ are two distinct vertices of $G$, and that $C$ is a component of $G-u-v$ to which both $u$ and $v$ do send some edges.
Then at least one of the following assertions holds:
\begin{enumerate}
\item $G[C+u+v]$ contains a $\TTT$ minor;
\item $G[C+u+v]$ contains a football minor connecting $u$ and $v$;
\item $G[C+u+v]$ contains an arrow barrage minor either from $u$ to $v$ or from $v$~to~$u$; in particular, $G[C+u+v]$ contains a muscle barrage minor connecting $u$ and $v$.
\end{enumerate}
\end{proposition}

\begin{proof}
We may assume that both $u$ and $v$ send infinitely many edges to $C$.
Indeed, if---say---$u$ sends only finitely many edges to $C$, 
then consider the \iecon\ graph $G':=G[C+v]$ and let $u'$ be one of the neighbours of $u$ in~$C$.
If there is a component $C'$ of $G'-u'-v$ to which both $u'$ and $v$ send infinitely many edges, then we may replace $G,u,v,C$ with $G',u',v,C'$.
Hence we may assume that there are infinitely many components $C'_0,C'_1,\ldots$ of $G'-u'-v$ such that, without loss of generality, $u'$ sends only finitely many but at least one edge to each $C'_n$ while $v$ sends infinitely many edges to each~$C'_n$.

By Theorem~\ref{finSepSTsExist}, all $\tilde{C}'_n$ have finitely separating spanning trees.
If one $\tilde{C}'_n$ has a finitely separating spanning tree that contains a subdivision of the infinite binary tree, then Lemma~\ref{fcutTreeContainsBinaryGivesTTT} provides a $\TTT$ minor witnessing~(i).
Otherwise, by Proposition~\ref{RecPrunableBinary}, every $\tilde{C}'_n$ has a rooted finitely separating spanning tree $T_n$ that is recursively prunable.
Then we pick for every $n$ a finite-degree node $t_n\in T_n$, and we let $P_n$ be a path in $C'_n$ that links a neighbour of $u'$ to the subgraph $C'_n[t_n]$ such that $P_n$ has only its endvertex $x_n$ in $C'_n[t_n]$.
Now we obtain an arrow barrage minor in $G[C+u+v]$ from $u$ to $v$ that is sought in~(iii), as follows.
For the branch set of the arrow barrage's nock we take $\{u,u'\}\cup \bigcup_{n\in\N}V(P_n\mathring{x}_n)$.
The arrows' payloads we let be modelled by the \iecon\ subgraphs~$C'_n[t_n]$ (see Lemma~\ref{basicFact}).
And for the branch set of the arrow barrage's head we take~$\{v\}$ (that $v$ sends infinitely many edges to each part $t_n$ is ensured by Lemma~\ref{basicFact} and the assumption that $u'$ sends only finitely many edges to each $C_n'$).

Therefore, we may assume that both $u$ and $v$ send infinitely many edges to $C$.
By Theorem~\ref{finSepSTsExist} we may let $T$ be a finitely separating spanning tree of $\tilde{C}$, rooted arbitrarily.
We make the following two observations.

If $T$ contains a subdivision of the infinite binary tree, then Lemma~\ref{fcutTreeContainsBinaryGivesTTT} yields a $\TTT$ minor giving~(i).

If $T$ has finite-degree nodes $t_u$ and $t_v$ (possibly $t_u=t_v$) such that $u$ sends infinitely many edges to the part $t_u\subset V(C)$ in $G$ and $v$ sends infinitely many edges to the part $t_v\subset V(C)$ in $G$, then we deduce~(ii), as follows.
By Lemma~\ref{basicFact} both $C[t_u]$ and $C[t_v]$ are \iecon .
If $t_u=t_v$, then $G[t_u+u]\cup G[t_v+v]$ is a football subgraph connecting $u$ and $v$.
Otherwise $t_u$ and $t_v$ are distinct.
Then we let $P$ be any $t_u$--$t_v$ path in~$C$, and $(G[t_u+u]\cup G[t_v+v]\cup P)/V(P)$ is a football minor connecting $u$ and $v$.

By these two observations and Proposition~\ref{RecPrunableBinary}, we may assume that $T$ is recursively prunable and that, without loss of generality, whenever $t\in T$ has finite degree then $v$ does send infinitely many edges to the part $t\subset V(C)$ in $G$ while $u$ may send only finitely many edges to it.

If $u$ sends edges in $G$ to infinitely many parts $t\in T$ that 
have finite degree in~$T$, then we find an arrow barrage minor from $u$ to $v$ giving~(iii), because 
$v$ sends infinitely many edges to all of the \iecon\ subgraphs $C[t]$ (cf.~Lemma~\ref{basicFact}) by our assumption above.
Otherwise $u$ sends, in total, only finitely many edges in $G$ to the parts $t\in T$ 
that have finite degree in~$T$.
Since $u$ sends infinitely many edges in $G$ to $C$, this means that we may assume without loss of generality that $u$ sends no edges to the parts $t\in T$ that have finite degree in~$T$.
Then Lemma~\ref{fcutTreeArrowBarrage} yields an arrow barrage minor from $u$ to $v$ giving~(iii).
\end{proof}

Now we have all we need to prove the main result of the section, Theorem~\ref{PowerHorse}.
In its proof, we will face the construction of a minor in countably many steps.
The following notation and lemma will help us to keep the technical side of this construction to the minimum.

Suppose that $G$ and $H$ are two graphs with $H$ a minor of $G$.
Then there are a vertex set $U\subset V(G)$ and a surjection $f\colon U\to V(H)$ such that the preimages $f^{-1}(x)\subset U$ form the branch sets of a model of $H$ in $G$.
A \emph{minor-map} $\varphi\colon G\rminor H$ formally is such a pair $(U,f)$.
Given $\varphi=(U,f)$ we address $U$ as $V(\varphi)$ and we write $\varphi=f$ by abuse of notation.
Usually, we will abbreviate `minor-map' as `map'.
If we are given two maps $\varphi\colon G\rminor H$ and $\varphi'\colon H\rminor H'$, then these give rise to another map $\psi\colon G\rminor H'$ by letting $V(\psi):=\varphi^{-1}(\varphi'^{\;-1}(V(H'))$ and $\psi:=\varphi'\circ (\varphi\rest V(\psi))$.
On the notational side we write $\varphi'\diamond\varphi=\psi$.

\begin{lemma}\label{limitMinor}
If $G_0,G_1,\ldots$ and $H_0\subset H_1\subset\cdots$ are sequences of graphs \mbox{$H_n\subset G_n$} with maps $\varphi_n\colon G_n\rminor G_{n+1}$ that restrict to the identity on $H_n$, then $G_0\rminor \bigcup_{n\in\N}H_n$.
\end{lemma}

\begin{proof}
Recursively, each map $\varphi_n\colon G_n\rminor G_{n+1}$ gives rise to a map $\hat{\varphi}_n\colon G_0\rminor G_{n+1}$ via $\hat{\varphi}_0:=\varphi_0$ and $\hat{\varphi}_{n+1}:=\varphi_{n+1}\diamond\hat{\varphi}_n$.
For every $n\in\N$ we write $V_x^n=\hat{\varphi}_n^{-1}(x)$ for all vertices $x\in H_{n+1}$.
For every vertex $x\in H:=\bigcup_{n\in\N}H_n$ we denote by $N(x)$ the least number $n$ with $x\in H_n$.
As the maps $\varphi_n$ restrict to the identity on $H_n$, for every vertex $x\in H$ the vertex sets $V_x^n$ form an ascending sequence $V_x^{N(x)}\subset V_x^{N(x)+1}\subset\cdots$ whose overall union we denote by $V_x$.
We claim that the vertex sets $V_x$ form the branch sets of an $H$ minor in~$G_0$.

Indeed, every branch set $V_x$ is non-empty and connected in $G_0$ because all $V_x^n$ are.
If $xy$ is an edge of $H$, then $G_0$ contains a $V_x^n$--$V_y^n$ edge as soon as $xy\in H_n$, and this edge is a $V_x$--$V_y$ edge due to the inclusions $V_x^n\subset V_x$ and $V_y^n\subset V_y$.
It remains to show that $V_x$ and $V_y$ are disjoint for distinct vertices $x,y\in H$.
This follows at once from the vertex sets $V_x^n$ and $V_y^n$ being disjoint for all $n$ and the definition of $V_x$ and $V_y$ as ascending unions of these vertex sets.
\end{proof}

Finally, we prove the main result of the section:

\begin{theorem}\label{PowerHorse}
Suppose that $G$ is any \iecon\ graph, that $u,v$ are two distinct vertices of $G$, and that $C$ is a component of $G-u-v$ to which both $u$ and $v$ do send some edges.
Then at least one of the following assertions holds:
\begin{enumerate}
\item $G[C+u+v]$ contains a $\TTT$ minor;
\item $G[C+u+v]$ contains a football minor connecting $u$ and $v$.
\end{enumerate}
\end{theorem}

\begin{proof}
Assume for a contradiction that both (i) and (ii) fail.
We will use Proposition~\ref{FindingAnArrowBarrage} to find the following graph $H$ as a minor in $G':=G[C+u+v]$.
Let $T_u$ be an $\aleph_0$-regular tree with root~$r_u$, and let $T_v$ be a copy of $T_u$ that is disjoint from $T_u$.
We write $r_v$ for the root of $T_v$.
The graph $H$ is obtained from the disjoint union of the two trees $T_u$ and $T_v$ by adding the perfect matching between their vertex sets that joins every vertex of $T_u$ to its copy in~$T_v$.
For every number $n\in\N$ we write $H_n$ for the subgraph of $H$ that is induced by the first $n$ levels of $T_u$ together with the first $n$ levels of~$T_v$.
Thus, $H=\bigcup_{n\in\N}H_n$. 
Finding an $H$ minor in $G'$ completes the proof, because $H/T_u$ is isomorphic to~$\TTT$.

A \emph{foresighted} $H_n$ is a graph that is obtained from $H_n$ by adding for every edge $xy\in H_n$ that runs between the two $n$th levels of $T_u$ and $T_v$ a muscle barrage $B_{xy}$ having endvertices $x$ and $y$ such that $B_{xy}$ contains no vertices from $H_n$ other than $x$ and $y$, and all muscle barrages added are pairwise disjoint.

By Lemma~\ref{limitMinor} it suffices to find a sequence $G'\rminor\hat{H}_0\rminor \hat{H}_1\rminor\cdots$ of graphs $\hat{H}_n$ that are foresighted $H_n$ with maps $\varphi_n\colon \hat{H}_n\rminor\hat{H}_{n+1}$ that restrict to the identity on $H_n\subset\hat{H}_n$ in order to find an $H$ minor in $\hat{H}_0\minor G'$.
To get started, we apply Proposition~\ref{FindingAnArrowBarrage} to $G,u,v,C$ to obtain in $G'$ a muscle barrage minor connecting $u$ and $v$.
By turning one of the muscles into an edge we obtain~$\hat{H}_0\minor G'$.

At step $n>0$, consider the muscle barrages $B_{xy}$ that turn $H_n$ into~$\hat{H}_n$.
For every muscle $M_{xy}^k$ of each of these muscle barrages $B_{xy}=\bigcup_{k\in\N}M_{xy}^k$ we apply Proposition~\ref{FindingAnArrowBarrage} in $M:=M_{xy}^k-x-y$ to the neighbours $x'$ and $y'$ of $x$ and $y$ in $M_{xy}^k$ and some component of $M-x'-y'$ to which both $x'$ and $y'$ send some edges to find a muscle barrage minor connecting $x'$ and $y'$.
By turning one muscle of each new barrage into an edge, we find $\varphi_n\colon \hat{H}_n\rminor\hat{H}_{n+1}$.
\end{proof}

\section{Proof of the main result}\label{sec:HalfTwo}

\noindent In this section we employ the main result of the previous section (Theorem~\ref{PowerHorse}) to prove the main result of this paper (Theorem~\ref{mainresult}).

\begin{lemma}\label{ABgiveS}
If $A$ and $B$ are two infinite vertex sets in a graph $G$ that does not contain a subdivision of $K^{\aleph_0}$, then there are vertices $a\in A$ and $b\in B$ plus a finite vertex set $S\subset V(G)\setminus\{a,b\}$ such that $S$ separates $a$ and $b$ in $G-ab$.
\end{lemma}
\begin{proof}
The absence of such an $S$ for a pair $a\neq b$ means that, inductively, we can find infinitely many independent $a$--$b$ paths in $G$.
So if there is no $S$ for every pair $a\neq b$, then inductively we find a $TK_{\aleph_0,\aleph_0}$ in $G$, and $TK^{\aleph_0}\subset TK_{\aleph_0,\aleph_0}$ (contradiction).
\end{proof}

\begin{lemma}\label{WlogThereIsS}
Suppose that $G$ is a football with endvertices $u$ and $v$.
If~$G$ does not contain a subdivision of $K^{\aleph_0}$, then $G$ contains an \iecon\ graph $H$ as a minor with branch sets $V_h$ $(h\in H)$ such that $u$ and $v$ are contained in distinct branch sets $V_x$ and $V_y$, respectively, and there is a finite vertex set $S\subset V(H)\setminus\{x,y\}$ separating $x$ and $y$ in~$H$.
\end{lemma}

\begin{proof}
Write $C$ for the \iecon\ graph $G-u-v$.
We apply Lemma~\ref{ABgiveS} in $C$ to the infinite neighbourhoods $N(u)$ and $N(v)$ of $u$ and $v$ in $G$ to obtain vertices $a\in N(u)$ and $b\in N(v)$ plus a finite vertex set $S\subset V(C)\setminus\{a,b\}$ that separates $a$ and $b$ in $C-ab$.
Then $H$ can be obtained from the \iecon\ graph $G-ab$ as follows.
We discard all the edges that are incident with $u$ or $v$, except for the two edges $ua$ and $vb$ each of which we contract.
Then $H$ is \iecon\ because it is isomorphic to the \iecon\ graph $C-ab$.
And the way we treated the edges at $u$ and $v$ ensures that $S$ separates the two vertices $\{u,a\}$ and $\{v,b\}$ in $H$ as desired.
\end{proof}

\begin{lemma}\label{HuAndHv}
Suppose that $G$ is an \iecon\ graph and that $u,v$ are two distinct vertices of $G$ that are separated in $G$ by some finite vertex set $S\subset V(G)\setminus\{u,v\}$.
Then there exist induced subgraphs $H_u,H_v\subset G$ containing $u$ and $v$ respectively, 
such that the following assertions hold:
\begin{enumerate}
    \item $X:=V(H_u)\cap V(H_v)$ is finite, non-empty and connected in $G$;
    \item both $H_u/X$ and $H_v/X$ are \iecon ;
    \item $X$ avoids $u$ and $v$;
    \item $uX$ is an edge of $H_u/X$ and $vX$ is an edge of $H_v/X$.
\end{enumerate}
\end{lemma}

\begin{proof}
Given $G,u,v,S$ let us write $C_u$ and $C_v$ for the distinct components of \mbox{$G-S$} that contain $u$ and $v$ respectively.
For both $w\in\{u,v\}$ we abbreviate $\sim_{G[C_w\cup S]}$ as $\sim_w$.
As $G$ is \iecon , we infer that every $\sim_w$-class meets $S$.
In particular, there are only finitely many $\sim_w$-classes in total, which means that each of the non-singleton classes induces an \iecon\ subgraph of $G$.
Let us write $K_u$ and $K_v$ for the \iecon\ subgraphs induced by the classes containing $u$ and $v$ respectively, i.e., $K_u:=G[\,[u]_{\sim_u}\,]$ and $K_v:=G[\,[v]_{\sim_v}\,]$.
To find $H_u$ and $H_v$, we distinguish two cases.

In the first case, $K_u$ and $K_v$ are disjoint.
For both $w\in\{u,v\}$, the finite partition of $V(C_w)\cup S$ induced by $\sim_w$ has only finitely many cross-edges. 
Since $G$ is \iecon , this means that we can find a $(K_u\cap S)$--$(K_v\cap S)$ path $P$ in $G$ avoiding all these finitely many edges.
Then $P$, as it may not use these edges, is a $K_u$--$K_v$ path with endvertices in $S$.
We let $P_w$ be a $w$--$P$ path in $K_w$ for both $w\in\{u,v\}$.
Letting $H_u:=G[K_u\cup P\cup\mathring{v}P_v]$ and $H_v:=G[K_v\cup P\cup\mathring{u}P_u]$ completes this case with $X=V(P_u\cup P\cup P_v)\setminus\{u,v\}$ because the graph $H_w/X$ contains the spanning subgraph $K_w/V(\mathring{w}P_w)$, and $K_w/V(\mathring{w}P_w)$ is \iecon\ by Lemma~\ref{finiteEdgeContractionOK} and because $K_w$ is \iecon .

In the second case, $K_u$ and $K_v$ meet in a vertex $s\in S$.
We write $D_u$ for the component of $K_u-u$ containing $s$.
In $D_u$ we pick a finite tree $T$ that contains the finite intersection $V(D_u)\cap V(K_v)\subset S$ and contains a neighbour of~$u$.
Then $T$ contains $s$ but neither $u$ nor $v$.
We let $P_v$ be any $v$--$s$ path in $K_v$.
Letting $H_u:=G[D_u\cup \mathring{v}P_v+u]$ and $H_v:=G[K_v\cup T]$ completes this case with~$X=V(T\cup \mathring{v}P_v)$:
On the one hand, the graph $H_u/X$ is \iecon\ because it contains the spanning subgraph $G[D_u+u]/V(T)$ which is \iecon\ by Lemma~\ref{finiteEdgeContractionOK} and the fact that $G[D_u+u]$ is an \iecon\ subgraph of~$K_u$.
On the other hand, the graph $H_v/X$ contains the spanning subgraph $K_v/Y$ for $Y:=(V(K_v)\cap V(D_u))\cup V(\mathring{v}P_v)$, and $K_v/Y$ is \iecon\ by Lemma~\ref{finiteEdgeContractionOK} and because $K_v$ is \iecon .
\end{proof}

\begin{definition}[Plows]
Suppose that $u$ and $v$ are two distinct vertices.
A \emph{half-plow} with \emph{endvertices} $u$ and $v$ is an \iecon\ graph containing the edge $uv$.
A \emph{plow} with \emph{endvertices} $u$ and $v$ and \emph{head} $h$ is a union of two half-plows with end-vertices $u,h$ and $h,v$ that do not meet in any vertex other than~$h$.
Plow minors \emph{connecting} some two vertices are defined like for footballs and muscles.
\end{definition}

\begin{theorem}\label{TTTorPlow}
If $G$ is an \iecon\ graph and $u,v$ are two distinct vertices of $G$, then at least one of the following two assertions holds:
\begin{enumerate}
\item $G$ contains a $\TTT$ minor;
\item $G$ contains a plow minor connecting $u$ and $v$.
\end{enumerate}
\end{theorem}

\begin{proof}
Let $G,u,v$ be given, we show $\neg$(i)$\to$(ii).
For this, suppose that $G$ does not contain a $\TTT$ minor.
By Theorem~\ref{PowerHorse} and Lemma~\ref{WlogThereIsS} we may assume that there is a finite vertex set $S\subset V(G)\setminus\{u,v\}$ that separates $u$ and $v$ in~$G$.
Then applying Lemma~\ref{HuAndHv} provides induced subgraphs $H_u,H_v\subset G$ containing $u$ and $v$ respectively, such that the following assertions hold:
\begin{itemize}[label={\textbf{--}}]
    \item $X:=V(H_u)\cap V(H_v)$ is finite, non-empty and connected in $G$;
    \item both $H_u/X$ and $H_v/X$ are \iecon ;
    \item $X$ avoids $u$ and $v$;
    \item $uX$ is an edge of $H_u/X$ and $vX$ is an edge of $H_v/X$.
\end{itemize}
Then $(H_u\cup H_v)/X$ is a plow minor connecting $u$ and~$v$.
\end{proof}

\begin{customthm}{\ref{mainresult}}
Every \iecon\ graph contains either the Farey graph or $\TTT$ as a minor.
\end{customthm}

\begin{proof}
If $G$ contains $\TTT$ as a minor, then we are done.
So let us suppose that $G$ does not contain a $\TTT$ minor.
Our task then is to find a Farey graph minor in~$G$.
By Lemma~\ref{FareyGisMinorOfHalvedFareyG} it suffices to find a halved Farey graph minor.

Call a graph a \emph{foresighted} halved Farey graph of \emph{order} $n\in\N$ if it is the union of $\hFG_n$ with \iecon\ graphs $A_{xy}$, one for every blue edge $xy\in \hFG_n$, such that:
\begin{enumerate}
    \item each $A_{xy}$ meets $\hFG_n$ precisely in $x$ and $y$ but $xy\notin A_{xy}$;
    \item every two distinct $A_e$ and $A_{e'}$ meet precisely in the intersection $e\cap e'$ of their corresponding edges (viewed as vertex sets).
\end{enumerate}

\noindent By Lemma~\ref{limitMinor} it suffices to find a sequence $H_0,H_1,\ldots$ of foresighted halved Farey graphs of orders $0,1,\ldots$ with maps $\varphi_n\colon H_n\rminor H_{n+1}$ that restrict to the identity on $\hFG_n\subset H_n$ to yield a halved Farey graph minor in~$G=:H_0$.

To get started, pick any edge $e$ of $G$, and note that $G=H_0$ is a foresighted halved Farey graph of order $0$ when we rename $e$ to the edge of which $\hFG_0=K^2$ consists.
At step $n+1$, suppose that we have already constructed $H_n\supset\hFG_n$, and consider the \iecon\ graphs $A_{xy}$ that were added to $\hFG_n$ to form~$H_n$.
Theorem~\ref{TTTorPlow} yields in each $A_{xy}$ a plow minor with head $h_{xy}$ that connects $x$ and $y$.
These plow-minors combine with $\hFG_n$ and with each other to give a map~$\varphi_n\colon H_n\rminor H_{n+1}\supset\hFG_{n+1}$ that sends the branch set of every head $h_{xy}$ to the vertex $v_{xy}\in\hFG_{n+1}-\hFG_n$ that arises from the blue edge $xy\in\hFG_n$ in the recursive definition of~$\hFG_{n+1}$.
\end{proof}

\section{Outlook}\label{sec:Outlook}

\noindent Here are two open problems that came to my mind.

\begin{problem}
Can Theorem~\ref{mainresult} be strengthened to always find one of the two minors with finite branch sets?
\end{problem}

Seymour and Thomas~\cite{ExcludingInfiniteTrees}, together with Robertson~\cites{ExcludingInfiniteMinors,ExcludingInfiniteCliques}, have characterised the graphs without $K^\kappa$ or $T_\kappa$ minors in terms of tree-decompositions and, alternatively, in terms of various other structures.
Can their list be extended to include the Farey graph?
Tree-decompositions might not be the right complementary structures for \iecon\ substructures, but there might be other structures (e.g.~$\fcut(G)$-trees):

\begin{problem}
Characterise the graphs without a Farey graph minor in terms of tree-decompositions or in terms of other structures.
\end{problem}

\section{Appendix}\label{sec:appendix}

\noindent The following lemma proves that Theorem~6.3 in~\cite{duality} is equivalent to Theorem~\ref{finSepSTsExist}.
The lemma and its proof are formulated in the terminology of~\cite{duality}.
In particular, $\tilde{G}$ denotes the topological space considered in~\cite{duality}, not the quotient graph that we considered in the previous sections.

\begin{lemma}\label{lem:theEquivalence}
Let $G$ be any finitely separable connected graph.
Then the following assertions are equivalent:
\begin{enumerate}
    \item $G$ has a spanning tree whose closure in $\tilde{G}$ contains no circle;
    \item $G$ has a finitely separating spanning tree.
\end{enumerate}
\end{lemma}

\begin{proof}
(ii)$\rightarrow$(i)
Every finite cut $F=E(V_1,V_2)$ of $G$ gives rise to a clopen bipartition $\overline{G[V_1]}\oplus\overline{G[V_2]}$ of the space~$\tilde{G}-\mathring{F}$, just like in the jumping arc lemma~\cite{DiestelBook5}*{8.6.3}.
Now suppose for a contradiction that $T\subset G$ is a finitely separating spanning tree and that $C\subset\overline{T}$ is a circle.
Then $C$ contains an edge $e\in T$ as Bruhn and Diestel remark in~\cite{duality}*{§2}.
But the fundamental cut $F_e$ of $e$ with respect to $T$ is finite, and hence its induced clopen bipartition topologically separates the endpoints of the arc $C-\mathring{e}$ in $\tilde{G}-\mathring{F}$, a contradiction.

(i)$\rightarrow$(ii)
Given any spanning tree $T\subset G$ whose closure in $\tilde{G}$ contains no circle, let us assume for a contradiction that some fundamental cut $F_e$ of $T$ is infinite.

We claim that
\begin{enumerate}[label=(\arabic*)]
    \item no ray in $T$ is dominated in~$G$, and that
    \item no two disjoint rays in $T$ are equivalent in~$G$.
\end{enumerate}
Indeed, if $T$ contains a ray that is dominated in $G$ by a vertex~$v$, then that ray is a tail of ray $R\subset T$ that starts in~$v$, so $\overline{R}\subset\overline{T}$ is a circle contradicting the choice of~$T$.
And if $T$ contains two disjoint equivalent rays, then there is a double ray $D\subset T$ that contains both rays, and neither of the two rays is dominated by~(1).
Thus, $\overline{D}\subset\overline{T}$ is a circle contradicting the choice of~$T$.

To complete the proof, we consider the two components $T_1$ and $T_2$ of $T-e$.

If some infinitely many edges in $F_e$ meet in the same vertex $v$ with $v\in T_1$, say, then applying the star-comb lemma in $T_2$ to their other endvertices must yield a comb since $G$ is finitely separable.
But then the spine of that comb is dominated by~$v$, contradicting~(1).

Otherwise we find an infinite independent edge set~$M\subset F_e$.
Applying the star-comb lemma in~$T_1$ to the endvertices of the edges in~$M$ yields either a star or a comb, and by replacing $M$ with an infinite subset we may assume without loss of generality that every edge in~$M$ has an endvertex that is either a leaf of the star or a tooth of the comb.
But now applying the star-comb lemma in $T_2$ to the endvertices of the edges in $M$ yields a contradiction, as follows.
On the one hand we cannot get a star, because this would contradict either that $G$ is finitely separable or~(1).
On the other hand we cannot get a comb, because this would contradict either (1) or~(2).
\end{proof}

\input{Fareybib.tex}
\end{document}

%% file: Preamble.tex
\usepackage[utf8]{inputenc} 
\usepackage{amssymb}
\usepackage{mathrsfs}
\usepackage{graphicx}
\usepackage{latexsym}
\usepackage{stmaryrd}
\usepackage{enumitem}
\usepackage[bookmarks,hyperfootnotes=false,backref]{hyperref}
\usepackage{multirow}
\usepackage{xcolor}
\usepackage{caption}
\colorlet{darkishRed}{red!80!black}
\colorlet{darkishBlue}{blue!60!black}
\colorlet{darkishGreen}{green!60!black}
\hypersetup{
    draft = false,
    bookmarksopen=true,
    colorlinks,
    linkcolor={red!60!black},
    citecolor={green!60!black},
    urlcolor={blue!60!black}
}
\usepackage[msc-links,backrefs]{amsrefs} 
\usepackage{doi}

\renewcommand{\PrintDOI}[1]{\doi{#1}}

\let\setminus=\smallsetminus
\usepackage{tikz}
\usepackage{tikz-cd}
\usetikzlibrary{calc,through,intersections,arrows, trees, positioning, decorations.pathmorphing, cd}
\usepackage{relsize}
\usepackage{comment}
\usepackage{svg}
\usepackage{mathtools}
\usepackage{nccmath}
\usepackage{pifont}

\let\setminus=\smallsetminus

\newcommand{\rest}{\upharpoonright}

\renewcommand{\subset}{\subseteq}
\renewcommand{\supset}{\supseteq}

\newcommand{\at}{attached to }

\newcommand{ \N } { \mathbb{N} }

\newcommand{ \Q } { \mathbb{Q} }

\newcommand{\dc}[1]{\lceil #1\rceil}
\newcommand{\uc}[1]{\lfloor #1\rfloor}

\makeatletter

\def\calCommandfactory#1{%
   \expandafter\def\csname c#1\endcsname{\mathcal{#1}}}
\def\frakCommandfactory#1{%
   \expandafter\def\csname frak#1\endcsname{\mathfrak{#1}}}

\newcounter{ctr}
\loop
  \stepcounter{ctr}
  \edef\X{\@Alph\c@ctr}
  \expandafter\calCommandfactory\X
  \expandafter\frakCommandfactory\X
  \edef\Y{\@alph\c@ctr}
  \expandafter\frakCommandfactory\Y
\ifnum\thectr<26
\repeat

\makeatother

\setenumerate{label={\normalfont (\roman*)},itemsep=0pt}

\def\lowfwd #1#2#3{{\mathop{\kern0pt #1}\limits^{\kern#2pt\raise.#3ex
\vbox to 0pt{\hbox{$\scriptscriptstyle\rightarrow$}\vss}}}}
\def\lowbkwd #1#2#3{{\mathop{\kern0pt #1}\limits^{\kern#2pt\raise.#3ex
\vbox to 0pt{\hbox{$\scriptscriptstyle\leftarrow$}\vss}}}}
\def\fwd #1#2{{\lowfwd{#1}{#2}{15}}}

\def\vS{{\hskip-1pt{\fwd S3}\hskip-1pt}}

\def\vE{{\hskip-1pt{\fwd{E}{3.5}}\hskip-1pt}}
\def\vF{{\hskip-1pt{\fwd{F}{3.5}}\hskip-1pt}}
\def\ve{\kern-1.5pt\lowfwd e{1.5}2\kern-1pt}
\def\ev{\kern-1pt\lowbkwd e{0.5}2\kern-1pt}
\def\vf{\kern-2pt\lowfwd f{2.5}2\kern-1pt}

\def\vs{\lowfwd s{1.5}1}
\def\sv{\lowbkwd s{0}1}
\def\vr{\lowfwd r{1.5}2}
\def\rv{\lowbkwd r02}
\def\vSd{{\mathop{\kern0pt S\lower-1pt\hbox{${}
     \scriptstyle'$}}\limits^{\kern2pt\raise.1ex
     \vbox to 0pt{\hbox{$\scriptscriptstyle\rightarrow$}\vss}}}}

\newtheorem{theorem}{Theorem}[section] 
\newtheorem{proposition}[theorem]{Proposition}

\newtheorem{lemma}[theorem]{Lemma}

\newtheorem{problem}[theorem]{Problem}
\newtheorem{mainresult}{Theorem}

\newtheorem{innercustomthm}{Theorem}

\newenvironment{customthm}[1]
  {\renewcommand\theinnercustomthm{#1}\innercustomthm}
  {\endinnercustomthm}

\theoremstyle{definition}

\newtheorem{definition}[theorem]{Definition}

\theoremstyle{remark}

\newtheorem*{ack}{Acknowledgement}

%% file: FareyGraphDef2.tex
The \emph{Farey graph} $F$ is the graph on $\Q\cup\{\infty\}$ in which two rational numbers $a/b$ and $c/d$ in lowest terms (allowing also $\infty=(\pm 1)/0$) form an edge if and only if $\det\bigl( \begin{smallmatrix}a & c\\ b & d\end{smallmatrix}\bigr)=\pm 1$, cf.~\cite{OfficeHoursGroupTheory}.
In this paper we do not distinguish between the Farey graph and the graphs that are isomorphic to it.
For our graph-theoretic proofs it will be more convenient to work with the following purely combinatorial definition of the Farey graph that is indicated in~\cite{OfficeHoursGroupTheory} and~\cite{hatcher2017topology}.

The \emph{halved Farey graph} $\hFG_0$ of order $0$ is a $K^2$ with its sole edge coloured blue.
Inductively, the \emph{halved Farey graph} $\hFG_{n+1}$ of order $n+1$ is the edge-coloured graph that is obtained from $\hFG_n$ by adding a new vertex $v_e$ for every blue edge $e$ of $\hFG_n$, joining each $v_e$ precisely to the endvertices of $e$ by two blue edges, and colouring all the edges of $\hFG_n\subset\hFG_{n+1}$ black.
The \emph{halved Farey graph} $\hFG:=\bigcup_{n\in\N}\hFG_n$ is the union of all $\hFG_n$ without their edge-colourings, and the \emph{Farey graph} is the union $F=G_1\cup G_2$ of two copies $G_1,G_2$ of the halved Farey graph such that $G_1\cap G_2=\hFG_0$.

\begin{lemma}\label{FareyGisMinorOfHalvedFareyG}
The halved Farey graph and the Farey graph are minor-equivalent.
\end{lemma}
\begin{proof}
The halved Farey graph is a subgraph of the Farey graph.
Conversely, the Farey graph is a minor of the halved Farey graph: if $e$ is a blue edge of $\hFG_1$, then the Farey graph is the contraction minor of $\hFG-e$ whose sole non-trivial branch set is~$V(\hFG_0)$, i.e., $(\hFG-e)/V(\hFG_0)\cong\FG$.
\end{proof}

We remark that the Farey graph is uniquely determined by its connectivity~\cite{FareyGraphChar}.